\documentclass[10pt,a4paper]{article}
\usepackage[english]{babel}
\usepackage{amsmath}
\usepackage{amssymb,amsbsy}
\usepackage{amsthm}
\usepackage[utf8]{inputenc}
\usepackage[T1]{fontenc}
\usepackage{enumerate}
\usepackage{mathrsfs}
\usepackage{bbm}
\usepackage{bm}

\usepackage{hyperref}

\usepackage{nicefrac}

\newtheorem{tw}{Theorem}[section]
\newtheorem{lm}[tw]{Lemma}
\newtheorem{wn}[tw]{Corollary}
\newtheorem{pr}[tw]{Proposition}

\theoremstyle{definition}
\newtheorem{df}{Definition}[section]
\newtheorem{uw}[tw]{Remark}

\newcommand{\R}{\mathbb{R}}
\newcommand{\Z}{\mathbb{Z}}
\newcommand{\N}{\mathbb{N}}
\newcommand{\T}{\mathbb{T}}

\newcommand{\cS}{\mathcal{S}}

\newcommand{\Q}{\mathbb{Q}}
\newcommand{\cA}{\mathcal{A}}

\newcommand{\cP}{\mathcal{P}}
\newcommand{\cT}{\mathcal{T}}

\newcommand{\vep}{\varepsilon}

\providecommand{\noopsort}[1]{} 

\title{Ratner's property and mild mixing for smooth flows on surfaces}
\author{Adam Kanigowski \and Joanna Ku\l{}aga-Przymus\thanks{Research supported by Narodowe Centrum Nauki grant DEC-2011/03/B/ST1/00407.}}

\begin{document}

\bibliographystyle{siam}
\maketitle

\begin{abstract}
Let $\cT=(T_t^f)_{t\in\R}$ be a special flow built over an IET $T:\T\to\T$ of bounded type, under a roof function $f$ with symmetric logarithmic singularities at a subset of discontinuities of $T$. We show that $\cT$ satisfies so-called switchable Ratner's property which was introduced in~\cite{FKa}. A~consequence of this fact is that such flows are mildly mixing (before, they were only known to be weakly mixing~\cite{MR2481331} and not mixing~\cite{MR2800723}). Thus, on each compact, connected, orientable surface of genus greater than one there exist flows which are mildly mixing and not mixing. 
\end{abstract}

\tableofcontents

\section{Introduction}
\subsection{Motivation}
This paper is concerned with measurable probability-preserving flows on standard Borel spaces. There are two notions which are of central interest. The first of them is so-called Ratner's property (originally, H-property~\cite{MR717825}, later named $H_p$-property or Ratner's property, see~\cite{MR1325699}), which describes a certain way of a divergence of orbits of nearby points, see Section~\ref{se:1.1}. We study this property for the first time in the literature in the class of special flows over interval exchange transformations. More specifically, we deal with so-called interval exchange transformations of bounded type. The roof function has so-called logarithmic singularities.

The special flows under consideration have a very natural origin. Namely, they arise when taking a cross section of smooth flows on compact, connected orientable surfaces. Such flows were studied already in the 1980s~\cite{MR676612}. However, even the very basic questions about properties such as weak mixing and mixing remained unanswered until very recently. In~\cite{MR2481331} a generic flow from this class was shown to be weakly mixing, whereas in~\cite{MR2800723} the absence of mixing was proved. Our motivation, apart from studying Ratner's property itself, was to discuss the question whether such flows enjoy the mild mixing property, intermediate between weak mixing and mixing. We give a positive answer to this question, provided that the base interval exchange transformation in the special flow representation is of bounded type. The key tool in our proof, which is also of independent interest, is Ratner's property.

\subsection{Ratner's property and its consequences}\label{se:1.1}
Ratner's property was first observed by M.~Ratner~\cite{MR717825} for the class of horocycle flows on the unit tangent bundles of compact surfaces of constant negative curvature.  In fact, instead of~\emph{the Ratner's property}, we should rather speak of \emph{Ratner's properties}, as this notion was later modified by various authors. Already in~\cite{MR796188}, D.~Witte extended the $H_p$-property to the so-called \emph{compact Ratner's property} and used it for studying the conjugacy problem for unipotent flows. The first examples of flows which were not of algebraic origin and were satisfying a slightly weakened Ratner's property, appeared in the literature several years later. K. Fr\k{a}czek and M. Lema\'{n}czyk~\cite{MR2237466} showed that so-called \emph{finite Ratner's property} holds for special flows over irrational rotations by $\alpha$, under piecewise absolutely continuous roof functions $f$ which satisfy von Neumann's condition $\int_{\mathbb{T}} f' \ d\lambda\neq 0$ (see~\cite{MR1503078}), whenever $\alpha$ has bounded partial quotients.\footnote{Notice that such flows are indeed different from the horocycle flows, as they are never mixing~\cite{MR0306629}, whereas that the horocycle flows are mixing of all orders, see~\cite{MR2186251}. In fact, such special flows are even spectrally disjoint with all mixing flows~\cite{FL04}.} Further examples include some special flows over rotations under piecewise constant roof functions~\cite{MR2342268}, special flows over two-dimensional rotations~\cite{MR2753947} (here the so called \emph{weak Ratner's property} has been introduced -- it is a weaker notion than the finite Ratner's property) and special flows over irrational rotations under some roof functions which are of bounded variation and are more general than piecewise absolutely continuous~\cite{Kanigowski:2013fk,Kanigowski:2013uq}. The most recent results of this flavor concern so-called Kochergin-type and Arnol'd-type (or Sinai-Khanin-type) flows (see~\cite{MR0516507}, ~\cite{MR1142204} and~\cite{MR1189019}) -- in~\cite{FKa} it is shown that such flows do not satisfy weak Ratner's property, but they do satisfy its further weakening, so-called \emph{switchable Ratner's property}, see Section~\ref{se:swr}.\footnote{In all of these classes, except for the Sinai-Khanin-type flows in~\cite{FKa}, an additional assumption that the rotation angle $\alpha$ (or both coordinates of the rotation angle in case of two-dimensional rotation) has bounded partial quotients, was made in order to prove that one of Ratner's properties holds.} This variant of Ratner's property is the one we deal with in this paper.

Given a flow, whether or not Ratner's property holds, is of an independent interest as, for example, finite Ratner's property is an isomorphism invariant. However, not less important are the strong dynamical consequences which can be derived from Ratner's properties. Namely, all mentioned Ratner's properties, except for the one considered by D. Witte, are designed in such a way that they imply certain rigidity of joinings (for the definition of joining, see Section~\ref{se:joi}). More precisely, a flow which is weakly mixing and additionally satisfies Ratner's property ($H_p$-property, finite or weak Ratner's property), automatically also enjoys the so-called~\emph{finite extension joinings property} (FEJ-property), i.e.\ for every ergodic flow $\cS$ acting on $(Y,\mathcal{C},\nu)$ and every ergodic joining $\rho\neq \mu\otimes \nu$ of $\cT$ and $\cS$, the flow $\cT\times \cS$ on $(X\times Y,\mathcal{B}\otimes \mathcal{C},\rho)$ is a finite extension of $\cS$ (this has been proved for various version of Ratner's property in the aforementioned works). Ratner's properties also imply the so-called \emph{pairwise independence property}: any pairwise independent self-joining of $\cT$ is automatically independent, i.e.\ it is the product measure.

Again, although FEJ-property being an isomorphism invariant, can be of interest in itself, its importance is also reflected in how it affects mixing properties of the flows under consideration. The central mixing property in this paper is~\emph{mild mixing}, first defined (for probability-preserving automorphisms) by H.~Furstenberg and B. Weiss~\cite{MR518553}. Recall that an automorphism $T$ of a standard probability Borel space $(X,\mathcal{B},\mu)$ is said to be mildly mixing if its Cartesian product with any ergodic conservative (finite or infinite measure preserving) automorphism remains ergodic.  Mild mixing is an intermediate property between weak mixing and mixing,  equivalent to absence of rigid factors~\cite{MR518553}, i.e.\ for no set $A\in\mathcal{B}$ with $\mu(A)\in(0,1)$ and no sequence $n_k\to\infty$ we have $\mu(T^{n_k}A\cap A)\to\mu(A)$. Similar definition is used and similar results hold also for flows. There is the following close relation between FEJ-property and mild mixing:
\begin{pr}[\cite{MR2237466}]\label{cons}
Let $\cT$ be a weakly mixing flow, which is not partially rigid. If $\cT$ additionally enjoys FEJ-property, then $\cT$ is mildly mixing.  
\end{pr}
An immediate consequence of the above discussion is that whenever $\cT$ is weakly mixing, not partially rigid and satisfies Ratner's property then it is mildly mixing. In~\cite{MR2237466,MR2753947,MR2342268,Kanigowski:2013fk,Kanigowski:2013uq} to prove mild mixing was in fact one of the main motivations for considering Ratner's property. An exception to this ``rule'' is~\cite{FKa} where the considered flows are already mixing and Ratner's property is used to show that they are mixing of all orders. Thus, in general, Ratner's property can be seen as a useful tool to ``enhance'' mixing properties of studied flows.

\subsection{Main result and its consequences}\label{se:1.2}
We deal with special flows over minimal interval exchange transformations (IETs). The roof function is of continuity class $C^2$, except for a finite number of points, all of them being some discontinuity points of the base IET. The singularities are of symmetric logarithmic type. For more details, see Section~\ref{se:dach}. Moreover, we make an additional assumption that the base IET is of bounded type, see Section~\ref{se:ba}. Our main result is the following:
\begin{tw}\label{main}
Let $\cT=(T_t^f)_{t\in\R}$ be a special flow built over an IET $T:\T\to\T$ of bounded type, under a roof function $f$ with the properties described above. Then $\cT$ satisfies switchable Ratner's property.
\end{tw}
\begin{uw}
In view of the results from~\cite{FKa}, we do not expect the special flows we consider to satisfy weak Ratner's property. Since the dynamical consequences of weak Ratner's property and switchable Ratner's property which we are interested in are the same, we concentrated on proving that the latter of these two properties holds, rather than disproving the first of them.
\end{uw}
\begin{uw}
Notice that so far Ratner's property, or some of its weaker versions, was established for special flows over irrational rotations (also horocycle flows fall into this scheme as they are loosely Bernoulli). Thus, Theorem~\ref{main} provides the first concrete examples of special flows with the base automorphism not being an isometry, not being continuous and satisfying a Ratner's property. 
\end{uw}

Special flows which we consider are natural representations of some smooth flows on surfaces. Let $\omega$ be a smooth closed differential 1-form on a closed compact orientable surface $M$ of genus $\mathbf{g}\geq 2$, equipped with a fixed smooth area form. Locally $\omega$, as a closed form, is given by $dH$ for some real-valued function $H$. This determines a flow $(\varphi_t)_{t\in\mathbb{R}}$ on $M$, given by the local solutions of the following system of differential equations: $\dot{x}=\partial H / \partial y,\ \dot{y}=-\partial H/\partial x$. The flow $(\varphi_t)_{t\in\mathbb{R}}$ is always area-preserving. Its special flow representation was derived in~\cite{Arnold91} and~\cite{MR0415681}, see also~\cite{Zorich94u}, and it is  of the form described in the beginning of this section. As any special flow representation, it arises by taking a cross section and considering the first return map -- it gives the \emph{base automorphism}, and the first return time -- it gives the \emph{roof function}. The study of such flows was originally motivated by questions from physics \cite{MR676612}. Later on, they became of some interest also from the point of view of dynamical systems. Crucial in this context is the following result, proved independently in various settings in~\cite{MR630878,MR0009485,MR1733872}: if $\omega$ is a Morse form (equivalently, if $H$ is a Morse function, i.e.\ it has a finite number of singular points, all of which are non-degenerated) then the surface can be decomposed into a finite number of periodic components (where all points are periodic) and a finite number of minimal components (where all forward and backward trajectories of the flow are dense). This, in turn, motivates an ergodic theoretical question, whether on each such minimal component the flow is ergodic and, if the answer is positive, what are its mixing properties. We additionally assume that the flow has no saddle connections (i.e.\ there are no orbits which contain both an incoming and an outgoing separatrix of a saddle). In such a case, $(\varphi_t)_{t\in\R}$ is known to be minimal~\cite{MR0009485}. 

Recall that ergodicity of a flow is equivalent to ergodicity of the base automorphism in its special flow representation. Moreover, mixing properties of special flows depend strongly on the properties of the roof function. Recall also that almost every special flow over an IET under a roof function with symmetric logarithmic singularities is known to be weakly mixing~\cite{MR2481331} and not mixing~\cite{MR2800723}. It is therefore natural to ask whether such flows satisfy the subtler property of being mildly mixing. We give the positive answer to this question, provided that the base IET is of bounded type:
\begin{wn}\label{wn3}
Let $\cT=(T_t^f)_{t\in\R}$ be a special flow built over an IET $T:\T\to\T$ of bounded type, under a roof function $f$ with symmetric logarithmic singularities at a subset of discontinuities of $T$. Then $\cT$ is mildly mixing.
\end{wn}
Moreover, as an immediate consequence we obtain the following:
\begin{wn}
On each compact, connected, orientable surface of genus greater than one there exist flows which are mildly mixing. 
\end{wn}
In view of Theorem~\ref{main} and the discussion in Section~\ref{se:1.1}, a natural path which can be taken in order to prove Corollary~\ref{wn3} involves showing that the flows under consideration are weakly mixing and not partially rigid. Weak mixing follows immediately from~\cite{MR2481331}, as the class of IETs considered there includes all IETs of bounded type. Moreover, under the additional assumption that the base IET satisfies so-called balanced partition lengths condition, such flows are not partially rigid~\cite{MR2974212}. In Section~\ref{se:ba} we show that for an IET this additional condition is equivalent to being of bounded type, thus making the proof of Corollary~\ref{wn3} complete.

The authors would like to thank Krzysztof Fr\k{a}czek, Mariusz Lema\'{n}czyk and Corinna Ulcigrai for their interest, valuable discussions and suggestions.
\section{Basic definitions}\label{se:def}
\subsection{Special flows}
Let $T:(X,\mathcal{B},\mu)\to(X,\mathcal{B},\mu)$ be an ergodic automorphism and let $f\in L^1(X,\mathcal{B},\mu)$ be a strictly positive function. One defines a $\Z$-cocycle by setting
$$
f^{(n)}(x)=\left\{\begin{array}{ccc}
f(x)+\ldots+f(T^{n-1}x) &\text{if} &n>0\\
0&\text{if} &n=0\\
-(f(T^nx)+\ldots+f(T^{-1}x))&\text{if} &n<0.\end{array}\right.
$$
A {\em special flow} $\cT=(T_t^f)_{t\in\R}$ over the \emph{base automorphism} $T$ under the \emph{roof function} $f$ is a flow acting on $(X^f,\mathcal{B}^f,\mu^f)$, where $X^f=\{(x,s) : x\in X, 0\leq s<f(x)\}$, whereas $\mathcal{B}^f$ and $\mu^f$ are the restrictions of $\mathcal{B}\otimes\mathcal{B}(\R)$ and $\mu\otimes \lambda$ to $X^f$ respectively ($\lambda$ stands for Lebesgue measure on $\R$). Under the action of $\cT$ points are moved vertically upward with unit speed, and we identify $(x,f(x))$ with $(Tx,0)$, i.e.\ for $(x,s)\in X^f$ we have
$$
T_t^f(x,s)=(T^nx,s+t-f^{(n)}(x)),
$$
where $n\in \Z$ is unique such that $f^{(n)}(x)\leq s+t<f^{(n+1)}(x)$. Moreover, if $X$ is a metric space with a complete metric $d$, so is $X^f$, with $d^f((x,s),(y,s'))=d(x,y)+|s-s'|$.
\subsection{Joinings}\label{se:joi}
$\cT=(T_t)_{t\in\R}$ and $\cS=(S_t)_{t\in\R}$ be two ergodic flows acting on $(X,\mathcal{B},\mu)$ and $(Y,\mathcal{C},\nu)$ respectively. We say that a measure $\rho$ on $(X\times Y,\mathcal{B}\otimes \mathcal{C})$ is a joining between $\cT$ and $\cS$ if $\rho$ is $\cT\times \cS$-invariant, $\rho|_{\mathcal{B}\otimes \{\emptyset,Y\}}=\mu$ and $\rho|_{\{\emptyset, X\}\otimes \mathcal{C}}=\nu$. We denote the set of such joinings by $J(\cT,\cS)$. The set of ergodic joinings is denoted by $J^e(\cT,\cS)$. A joining $\rho\in J(\cT,\cS)$ is called a {\em finite extension} of $\nu$ if the natural projection $\pi:(X\times Y,\mathcal{B}\otimes \mathcal{C},\rho, \cT\times\cS)\to (Y,\mathcal{C},\nu,\cS)$ is finite to one. Let $\{A_n: n\in\N \}\subset \mathcal{B}, \{B_n: n\in\N\}\subset \mathcal{C}$  be two countable families, dense in $\mathcal{B}$ and $\mathcal{C}$ for the pseudometrics $d_\mu(A,A')=\mu(A\triangle A')$ and $d_\nu(B,B')=\nu(B\triangle B')$ respectively. Then $J(\cT,\cS)$ endowed with the metric $d$ given by 
$$
d(\rho,\rho')=\sum_{m,n\in\N}\frac{1}{2^{m+n}}|\rho(A_n\times B_m)-\rho'(A_n\times B_m)|
$$
is compact. We will refer to the corresponding topology as the weak topology. For more information about the theory of joinings, we refer the reader e.g. to~\cite{MR1958753}.

Let $(t_n)_{n\in\N}\subset \R$ be such that $t_n\to \infty$. Recall that $\cT$ is said to be~\emph{rigid} along $(t_n)$ if $\mu(T_{t_n}A\cap A)\to \mu(A)$ for all $A\in\mathcal{B}$. In other words, $\cT$ is rigid along $(t_n)$ if $\mu_{T_n}\to \mu_{Id}$ weakly, where for $R\in C(\cT)$, $\mu_R\in J(\cT,\cT)$ is given by $\mu_R(A\times B)=\mu(RA\cap B)$.\footnote{$C(\cT)$ stands for the centralizer of the flow $\cT$.} Moreover, $\cT$ is called \emph{partially rigid} along $(t_n)$ if there exists $a\in (0,1]$ such that for every $A\in\mathcal{B}$, $\liminf_{n\to \infty}\mu(T_{t_n}A\cap A)\geq \mu(A)$. Partial rigidity also can be expressed in terms of joinings: $\cT$ is partially rigid along $(t_n)$ if and only if for any $(n_k)_{k\in \N}$ such that the sequence $(T_{t_{n_k}})_{k\in\N}$ converges in the weak operator topology, we have $\lim_{k\to\infty}\mu_{T_{t_{n_k}}}(A\times B)=a\mu_{Id}(A\times B)+(1-a)\rho(A\times B)$ for some $\rho\in J(\cT,\cT)$.


\subsection{Interval exchange transformations}\label{se:iet}
Recall that an \emph{interval exchange transformation} (IET) is a piecewise orientation preserving isometry of a finite interval to itself. It is determined by its combinatorial data and length data in the following way.\footnote{We use the notation introduced in~\cite{MMY05}. For more information on interval exchange transformations, we refer the reader e.g. to~\cite{VianaBook,Yoccoz06}} Let $\cA$ be a finite alphabet with $d\geq 1$ letters. The~\emph{combinatorial data} of an IET is a pair of bijections $(\pi_0,\pi_1)$, $\pi_\vep\colon \cA\to\{1,\dots,d\}$ for $\vep=0,1$. The \emph{length data} of an IET is a vector $\lambda=(\lambda_\alpha)_{\alpha\in\cA}\in\R_+^\cA$, where $\R_+=(0,+\infty)$. We set $|\lambda|:=\sum_{\alpha\in\cA}\lambda_\alpha$, $|I|=[0,|\lambda|)$ and
\begin{align*}
&I_\alpha:=[\ell_\alpha,r_\alpha),\text{ where }\ell_\alpha=\sum_{\pi_0(\beta)<\pi_0(\alpha)}\lambda_\beta\text{ and } r_\alpha=\sum_{\pi_0(\beta)\leq\pi_0(\alpha)}\lambda_\beta,\\
&I_\alpha':=[\ell_\alpha',r_\alpha'),\text{ where }\ell_\alpha'=\sum_{\pi_1(\beta)<\pi_1(\alpha)}\lambda_\beta\text{ and } r_\alpha'=\sum_{\pi_1(\beta)\leq\pi_1(\alpha)}\lambda_\beta,
\end{align*}
Then clearly, $|I_\alpha|=\lambda_\alpha$ for $\alpha\in\cA$. The IET $T=T_{\pi,\lambda}\colon [0,|\lambda|)\to[0,|\lambda|)$ defined by $(\pi,\lambda)$ translates $I_\alpha$ to $I_\alpha'$ for each $\alpha\in\cA$. The endpoints of $I_\alpha$ are called the \emph{discontinuities} of $T$. We set $\Sigma_{\pi,\lambda}=\Sigma_T:=\{\ell_\alpha,r_\alpha:\alpha\in\cA\}$. Recall that that $T$ is said to be \emph{irreducible} if
$$
\pi_0^{-1}(\{1,\dots,j\})\neq\pi_1^{-1}(\{1,\dots,j\}) \text{ for }1\leq j<d.
$$
We say then that the pair $(\pi_0,\pi_1)$ is \emph{admissible}. Recall also that $T$ satisfies the \emph{Keane condition} if the orbits of $\ell_\alpha$ with $\pi_0(\alpha)\neq 1$ are infinite and disjoint. This condition implies minimality~\cite{Keane75}.

A special case of IETs are irrational rotations on the circle. Given $\alpha\not\in\Q$, we denote by $T_\alpha x=x+\alpha$ the corresponding
irrational rotation on $(\mathbb{T}, \mathcal{B}(\mathbb{T}), \lambda)$. The circle $\mathbb{T}=\mathbb{R}/\mathbb{Z}$ is identified with the interval $[0,1)$, the measure $\lambda$ is the Lebesgue measure inherited from $[0,1)$. Rotation on the circle is an exchange of two intervals. By $(q_n)_{n\geq 0}$ we denote the sequence of denominators of $\alpha$ in its continued fraction expansion $[a_0;a_1,a_2,\dots]$, i.e.\ we have
$$
q_0=1,\ q_1=a_1,\ q_{n+1}=a_{n+1}q_n+q_{n-1}.
$$

\subsubsection{Rauzy-Veech induction}
Let $T=T_{\pi,\lambda}$ with $\pi$ irreducible be an IET exchanging $d$ intervals, satisfying the Keane condition. Then $\lambda_{\pi_0^{-1}(d)}\neq \lambda_{\pi_1^{-1}(d)}$ and we set 
$$
\widetilde{I}:=\left[0,\max\left(\ell_{\pi_0^{-1}(d)},\ell_{\pi_1^{-1}(d)}\right)\right)
$$
and denote by $\mathcal{R}(T)=\widetilde{T}\colon \widetilde{I}\to \widetilde{I}$ the first return map of $T$ to the interval $\widetilde{I}$. Then $\widetilde{T}$ is again an IET exchanging $d$ intervals~\cite{Rauzy79}. More precisely, $\widetilde{T}$ is given by the following combinatorial and length data. Let
$$
\vep(\pi,\lambda):=\begin{cases}
0& \text{if }\lambda_{\pi_0^{-1}(d)}>\lambda_{\pi_1^{-1}(d)}\\
1& \text{if }\lambda_{\pi_0^{-1}(d)}<\lambda_{\pi_1^{-1}(d)}.
\end{cases}
$$
In either case, the longer of the two intervals $I_{\pi_0^{-1}(d)}$ and $I_{\pi_1^{-1}(d)}$ is called the \emph{winner}. The pair $\widetilde{\pi}=R_\vep(\pi_0,\pi_1)=(\widetilde{\pi}_0,\widetilde{\pi}_1)$ is given by
\begin{align*}
\widetilde{\pi}_\vep(\alpha)&:=\pi_\vep(\alpha)\text{ for every }\alpha\in\mathcal{A},\\
\widetilde{\pi}_{1-\vep}(\alpha)&:=\begin{cases}
\pi_{1-\vep}(\alpha)&\text{ if }\pi_{1-\vep}(\alpha)\leq \pi_{1-\vep}\circ \pi_\vep^{-1}(d)\\
\pi_{1-\vep}(\alpha)+1&\text{ if }\pi_{1-\vep}\circ \pi_\vep^{-1}(d)<\pi_{1-\vep}(\alpha)<d\\
\pi_{1-\vep}\circ\pi_\vep^{-1}(d)+1&\text{ if }\pi_{1-\vep}(\alpha)=d.
\end{cases}
\end{align*}
Finally, $\widetilde{\lambda}=\Theta^{-1}(\pi,\lambda)\lambda$, where
$$
\Theta(T)=\Theta(\pi,\lambda):=I+E_{\pi_\vep^{-1}(d),\pi_{1-\vep}^{-1}(d)}\in SL(\Z^\cA)
$$
(for $\alpha,\beta\in\mathcal{A}$ the matrix $E_{\alpha,\beta}$ has only one non-zero entry: the value at position $(\alpha,\beta)$ is equal to $1$). 

Recall that if $T$ satisfies the Keane condition, so does $\widetilde{T}$. This means that the above procedure can be iterated, giving a sequence of IETs $(\mathcal{R}^n(T))_{n\geq 0}$. Denote by $\pi^n=(\pi_0^n,\pi_1^n)$, $\lambda^n=(\lambda_\alpha^n)_{\alpha\in\mathcal{A}}$ the combinatorial data and the length data defining $\mathcal{R}^n(T)$. Then $\mathcal{R}^n(T)$ is the first return map of $T$ to the interval $I^n=[0,|\lambda^n|)$. Moreover,
\begin{equation}\label{eq:odcinki}
\lambda^{n-1}=\Theta(\mathcal{R}^{n-1}(T))\lambda^{n},
\end{equation}
whence
$$
\lambda=\Theta^{(n)}(T)\lambda^n, \text{ where }\Theta^{(n)}(T)=\Theta(T)\cdot \Theta(\mathcal{R}(T))\cdot\ldots\cdot \Theta(\mathcal{R}^{n-1}(T)).
$$
Finally, let $I^n_\alpha$, $\alpha\in\mathcal{A}$ be the intervals exchanged by $\mathcal{R}^n(T)$.

\subsubsection{Acceleration of Rauzy-Veech induction}
Let $T$ be an IET satisfying the Keane condition. Given an increasing sequence $(n_k)_{k\geq 0}\subset \N$ with $n_0=0$, one can define an acceleration of the Rauzy induction algorithm in the following way. Let
\begin{equation}\label{koc}
B(n_k,n_{k+1}):=\Theta(\mathcal{R}^{n_k}(T))\cdot\Theta(\mathcal{R}^{n_k+1}(T))\cdot\ldots\cdot \Theta(\mathcal{R}^{n_{k+1}-1}(T)).
\end{equation}
Then, for any $k<k'$,
$$
\lambda^{n_k}=h(n_k,n_{k'})\lambda^{n_{k'}},
$$
where 
$$
h(n_k,n_{k'})=B(n_k,n_{k+1})B(n_{k+1},n_{k+2})\cdot\ldots\cdot B(n_{k'-1},n_{k'}).
$$
We will write $h^{n_k}$ for $h(0,n_k)$. By the definition, $\mathcal{R}^{n_{k'}}(T)\colon I^{n_{k'}}\to I^{n_{k'}}$ is the first return map of $\mathcal{R}^{n_k}(T)\colon I^{n_k}\to I^{n_k}$ to the interval $I^{n_{k'}}\subset I^{n_k}$. Moreover, $h_{\alpha,\beta}(k,k')$ is the time spent by any point from $I_\beta^{n_{k'}}$ in $I_{\alpha}^{n_k}$ until it returns to $I^{n_{k'}}$. Therefore $h_\beta(n_k,n_{k'}):=\sum_{\alpha\in\mathcal{A}}h_{\alpha,\beta}(n_k,n_{k'})$ is the first return time of points of $I_\beta^{n_{k'}}$ to $I^{n_{k'}}$. Notice that this quantity does not depend on $k$. We will therefore write $h_\beta^{n_{k'}}$ for $h_\beta(n_k,n_{k'})$.

There are two particular types of acceleration of the Rauzy induction algorithm which are interested in. The first of them was considered by Zorich~\cite{Zorich96}, by taking $n_0=0$ and $n_{k+1}=n(\pi^{n_k},\lambda^{n_k})$, where 
$$
n(\pi,\lambda)=\min\{k\geq 1 : \vep(\pi^k,\lambda^k)=1-\vep(\pi,\lambda)\}.
$$
A further acceleration of the Rauzy induction algorithm was defined by Marmi, Moussa and Yoccoz in~\cite{MMY05}. Before we give the details, we need to recall the notion of the \emph{Rauzy diagram}. It is a diagram associated to the Rauzy induction algorithm, whose vertices are admissible pairs $(\pi_0,\pi_1)$. Each vertex $(\pi_0,\pi_1)$ is the starting point of two arrows with endpoints $R_0(\pi_0,\pi_1)$ and $R_1(\pi_0,\pi_1)$. We say that an arrow $\gamma$ in the Rauzy diagram \emph{takes the name} $\alpha\in\mathcal{A}$ if $I_\alpha$ is the winner for this induction step. For an IET $T$ satisfying the Keane condition let $\gamma^{(n)}$ be the arrow in the Rauzy diagram connecting $(\pi_0^{(n-1)},\pi_1^{(n-1)})$ to $(\pi_0^{(n)},\pi_1^{(n)})$. Now fix $1\leq \widetilde{d}<d$. We define inductively an increasing sequence $n_{\widetilde{d},k}$ by setting $n_{\widetilde{d},0}:=0$ and letting $n_{\widetilde{d},k+1}$ be the largest integer such that no more than $\widetilde{d}$ names are taken by the arrows $\gamma^{(n)}$ for $n_{\widetilde{d},k}< n\leq n_{\widetilde{d},k+1}$. Clearly, for $1<\widetilde{d}<d$, $(n_{\widetilde{d},k})_{k\geq 0}$ is a subsequence of $(n_{\widetilde{d}-1,k})_{k\geq 0}$. The case $\widetilde{d}=1$ corresponds to the acceleration considered by Zorich described above, whereas in~\cite{MMY05} the emphasis was put on the case $\widetilde{d}=d-1$. For simplicity, instead of writing $(n_{d-1,k})_{k\geq 0}$ we will write $(m_k)_{k\geq 1}$. We will refer to  $B(m_k,m_{k+1})$ as the \emph{Marmi-Moussa-Yoccoz} (MMY) cocycle of $T$.

\section{IETs of bounded type}\label{se:ba}
By the norm of a matrix (or a vector) we will mean the largest absolute value of the coefficients, i.e.\ for $B=(B_{\alpha,\beta})_{\alpha,\beta\in\mathcal{A}}$ we set $\|B\|:=\max_{\alpha,\beta\in\mathcal{A}}|B_{\alpha,\beta}|$. Recall that there are several ways to define \emph{IETs of bounded type}, see~\cite{Hubert:2012fk} and~\cite{MR3178778}. We will use the definition which is given in terms of the MMY cocycle. We will also use the notion of an IET with~\emph{balanced partition lengths} which was introduced in~\cite{MR2974212}.
\begin{df}\label{de:1}
An IET $T$ is said to be of \emph{bounded type} if the MMY cocycle of $T$ is bounded, i.e.\ for some $C>0$
\begin{equation}\label{mmy}
\|B(m_k,m_{k+1})\|\leq C\text{ for every }k\in\N.
\end{equation}
\end{df}
Given a matrix $A\in SL(d,\Z)$ with strictly positive entries, following~\cite{Veech81}, we set
\begin{align*}
\nu_1(A)&:=\max\{{A_{\alpha\gamma}}/{A_{\beta\gamma}}:\alpha,\beta,\gamma\in\cA\},\\
\nu_2(A)&:=\max\{{A_{\gamma\alpha}}/{A_{\gamma\beta}}:\alpha,\beta,\gamma\in\cA\},\\
\nu(A)&:=\max\{\nu_1(A),\nu_2(A)\}.
\end{align*}
Notice that $\nu(A)\leq \|A\|$. Then, since
$$
\lambda^{m_k}=B(m_k,m_{k+1})\lambda^{m_{k+1}}\text{ and }h^{m_{k+1}}=h^{m_{k}}B(m_{k},m_{k+1}),
$$
for all $\alpha,\beta\in\cA$ we have
\begin{align}
\begin{split}\label{balancedtimes}
\frac{1}{C}|I_\beta^{m_k}|&\leq |I_\alpha^{m_k}|\leq C|I_\beta^{m_k}|,\\
\frac{1}{C}h_\beta^{m_k}&\leq h_\alpha^{m_k}\leq Ch_\beta^{m_k}.
\end{split}
\end{align}
Therefore, 
\begin{equation}\label{pigeon}
\frac{|I^0|}{dC^2}\leq|I_\alpha^{m_k}|\cdot h_\alpha^{m_k}\leq |I^0| \text{ for all }\alpha\in\cA.
\end{equation}
Moreover, for all $\alpha,\beta\in\cA$,
\begin{align}
\begin{split}\label{wys0}
h_\beta^{m_{k+1}}&=\sum_{\alpha,\gamma\in\cA}h_{\alpha,\gamma}^{m_k}h_{\gamma,\beta}(m_k,m_{k+1})\leq C\sum_{\gamma\in\cA}\sum_{\alpha\in\cA}h_{\alpha,\gamma}^{m_k}\\
&=C\sum_{\gamma\in\cA}h_{\gamma}^{m_k}\leq dC^2\cdot h_{\alpha}^{m_k},
\end{split}
\end{align}
\begin{equation}
h_\beta^{m_{k+1}}=\sum_{\alpha,\gamma\in\cA}h_{\alpha,\gamma}^{m_k}h_{\gamma,\beta}(m_k,m_{k+1})\geq \sum_{\alpha,\gamma\in\cA}h_{\alpha,\gamma}^{m_k}= \sum_{\gamma\in\cA}h_{\gamma}^{m_k}\geq \frac{d}{C}\cdot h_\alpha^{m_k}\label{wys}
\end{equation}
and
\begin{align*}
|I_\beta^{m_{k}}|&=\sum_{\gamma\in\cA}h_{\gamma,\beta}(m_k,m_{k+1})|I_\gamma^{m_{k+1}}|\leq C\sum_{\gamma\in\cA}|I_{\gamma}^{m_{k+1}}|\leq dC^2\cdot |I_\alpha^{m_{k+1}}|,\\
|I_\beta^{m_k}|&=\sum_{\gamma\in\cA} h_{\gamma,\beta}(m_k,m_{k+1})|I_\gamma^{m_{k+1}}|\geq \sum_{\gamma\in\cA}|I_{\gamma}^{m_{k+1}}|\geq \frac{d}{C}\cdot |I_\alpha^{m_{k+1}}|.
\end{align*}

Given a partition $\mathcal{P}$ of an interval into subintervals we denote by $\min \cP$ and $\max \cP$ the smallest and the largest length of the atoms in these partitions respectively. Moreover, for a finite set $A\subset [0,1)$ we denote by $\cP(A)$ the partition of $[0,1)$ determined by $A$. For $\alpha\in\cA$, $n\in\N$ and $j\in\Z$ let $\mathcal{P}_{n,j}^\alpha$ and $\mathcal{P}_{n,j}$ be the partitions of the interval $[0,|\lambda|)$ defined in the following way:
$$
\cP_{n,j}^\alpha:=\cP(\{T^{-k+j}\ell_\alpha : 0\leq k\leq n-1\})\text{ and }\cP_{n,j}:=\bigcup_{\alpha\in\cA}\cP_{n,j}^{\alpha}.
$$
\begin{df}[cf. \cite{MR2974212}]\label{de:2}
We say that the interval exchange transformation $T=T_{\pi,\lambda}$ has \emph{balanced partition lenghts} whenever there exists $c>0$ such that for any $n\in\N$ the following two conditions hold:
\begin{enumerate}[(i)]
\item
for any $\alpha\in\cA$ and any $0\leq j\leq n-1$, we have
$$
\frac{1}{cn} <\min \cP_{n,j}^\alpha \leq \max \cP_{n,j}^\alpha <\frac{c}{n},
$$
\item\label{2}
for any $0\leq j\leq n-1$, we have
$$
\frac{1}{cn} <\min \cP_{n,j} \leq \max \cP_{n,j} <\frac{c}{n}.
$$
\end{enumerate}
\end{df}
\begin{uw}
In the original definition of balanced partition lenghts in~\cite{MR2974212} the quantifiers were different: in~\eqref{2} instead of $0\leq j \leq n-1$ only $j=0$ was considered.
\end{uw}

Recall that in case of irrational rotations, i.e.\ IETs of two intervals the following are equivalent:
\begin{itemize}
\item
$\alpha$ has bounded partial quotients, i.e.\ the exists $M>0$ such that $a_n<M$ for all $n$, where $[a_0;a_1,a_2,\dots]$ is the continued fraction expansion of $\alpha$,
\item
the associated IET is of bounded type,
\item
the associated IET has balanced partition lengths.
\end{itemize}
Also, in case of IETs of more than two intervals there is a relation between the notions of \emph{being of bounded type} and \emph{having balanced partition lenghts}:
\begin{pr}[\cite{Hubert:2012fk}: Proposition 1.1 and Theorem 4.7 and~\cite{MR3178778}: Section 2]\label{rowy}
Let $T$ be an IET. Then $T$ is of bounded type if and only if there exists $c>0$ such that
$$
 \frac{1}{cn}\leq \min \mathcal{P}_{n,0} \text{ for all }n\in\N.
$$
\end{pr}
As an immediate consequence of the above result we obtain the following:
\begin{wn}\label{rowy1}
Let $T$ be an IET. If $T$ has balanced partition lengths then $T$ is of bounded type.
\end{wn}
We will now prove the following strengthening of Proposition~\ref{rowy}, which is the converse of Corollary~\ref{rowy1}:
\begin{pr}\label{as} 
Let $T$ be an IET. Then $T$ is of bounded type if and only if $T$ has balanced partition lengths.
\end{pr}
\begin{proof}
Assume $T$ is of bounded type. It follows (see~\cite{MR3178778}) that for $\tilde{r}=\max(2d-3,2)$ all entries of the matrix $B(m_k,m_{k+\tilde{r}})$ are positive for all $k\in\N$. Let $r'\in \N$ be such that for every $k\in \N$, 
\begin{equation}\label{mina}\min_{\alpha\in \cA}|I_\alpha^{m_k}|\geq |I^{m_{k+r'}}|
\end{equation}
and let $r=\max(\tilde{r},r')$. We will show that there exists $D>0$ such that for every $n\in\N$ and $0\leq j\leq n-1$
\begin{equation}\label{em}
\frac{1}{Dn}\leq \min\cP_{n,j}.
\end{equation}
We claim that is suffices to show that for some $D>0$
\begin{equation}\label{hmk}
\frac{dC^2}{D\min_{\alpha\in\cA}h_\alpha^{m_k}}\leq \min\cP_{\min_{\alpha\in\cA}h_\alpha^{m_k},j}
\end{equation}
for every $k\in \N$ and $0\leq j\leq \min_{\alpha\in \cA}h_\alpha^{m_k}-1$. Indeed, assume that~\eqref{hmk} holds, fix $n\in \N$ and let $k\in \N$ be such that 
$$
\min_{\alpha\in\cA}h_\alpha^{m_k}\leq n<\min_{\alpha\in\cA}h_\alpha^{m_{k+1}}.
$$
It follows from~\eqref{hmk} and~\eqref{wys0} that, for every $0\leq j\leq n-1<\min_{\alpha\in\cA}h_\alpha^{m_{k+1}}-1$, we have
\begin{multline*}
\min\cP_{n,j}\geq\min\cP_{\min_{\alpha\in\cA}h_\alpha^{m_{k+1}},j}\\
 \geq  
 \frac{dC^2}{D\min_{\alpha\in\cA}h_\alpha^{m_{k+1}}}\geq \frac{dC^2}{dC^2D\min_{\alpha\in\cA}h_\alpha^{m_{k}}}\geq \frac{1}{Dn}.
\end{multline*}

We will now show that \eqref{hmk} indeed holds. Fix $k\in\N$. We claim that for every $\alpha\in \cA\setminus\{T^{-1}(\pi_0^{-1})(1)\}$,
\begin{equation}\label{orb1} 
\{l_\alpha,\ldots, T^{\min_{\alpha\in\cA}h_\alpha^{m_k}}l_\alpha\}\cap I^{m_{k+r}}=\emptyset, 
\end{equation}
and for $\alpha\in \cA\setminus\{\pi_0^{-1}(1)\}$,
\begin{equation}\label{orb2} \{T^{-\min_{\alpha\in\cA}h_\alpha^{m_k}}l_\alpha,\ldots, l_\alpha\}\cap I^{m_{k+r}}=\emptyset.
\end{equation}
Fix $\alpha \in \cA\setminus\{T^{-1}(\pi_0^{-1})(1)\}$. It follows by~\eqref{mina} that there exists $\beta\in\cA$ such that $I^{m_k}_\beta\cap I^{m_{k+r}}=\emptyset$. Moreover, since $l_\alpha$ is the left end of some level of some tower for $\mathcal{R}^{m_{k+r}}(T)$, the forward orbit of $l_\alpha$ visits $I^{m_k}_\beta$ before it gets to $I^{m_{k+r}}$.  Therefore $T^il_\alpha\notin I^{m_{k+r}}$ for $i=0,\ldots,\min_{\alpha\in\cA}h_\alpha^{m_k}$. This gives us~\eqref{orb1}. To show~\eqref{orb2}, we proceed similarly. Namely, for  $\alpha \in \cA\setminus\{\pi_0^{-1}(1)\}$, consider the backward orbit of $l_\alpha$, and let $i_\alpha\in \N$ be the smallest integer such that, $T^{-i_\alpha}l_{\alpha}\in I^{m_k}$ (this is some discontinuity point of $R^{m_k}(T)$). Hence, using again \eqref{mina}, there exists $\beta\in\cA$ such that $$T^{-i_\alpha}l_{\alpha}\in I^{m_k}_\beta\text{ and }I^{m_k}_\beta\cap I^{m_{k+r}}=\emptyset.$$ 
Therefore, for $\ell\in \{0,...,\min_{\alpha\in\cA}h^{m_k}_\alpha\}$, $T^{-\ell}l_\alpha \notin I^{m_{k+r}}$, which yields~\eqref{orb2}. Clearly, 
\begin{align*}
\cP_{\min_{\alpha\in\cA}h_\alpha^{m_k},j}=&\bigcup_{\alpha\in\cA}
\cP^\alpha_{\min_{\alpha\in\cA}h_\alpha^{m_k},j}\\
\subset&\bigcup_{\alpha\in\cA\setminus T^{-1}\pi_0^{-1}(1)} \{l_\alpha,\ldots, T^{\min_{\alpha\in\cA}h_\alpha^{m_k}}l_\alpha\}\\
&\cup\bigcup_{\alpha\in\cA\setminus \pi_0^{-1}(1)} \{T^{-\min_{\alpha\in\cA}h_\alpha^{m_k}}l_\alpha,\ldots, l_\alpha\}\\
=&\bigcup_{\alpha\in\cA} \{T^{-\min_{\alpha\in\cA}h_\alpha^{m_k}}l_\alpha,\ldots,T^{\min_{\alpha\in\cA}h_\alpha^{m_k}}l_\alpha\}.
\end{align*}
Moreover, it follows from~\eqref{orb1} and~\eqref{orb2} that all points from the set 
$$
\bigcup_{\alpha\in\cA} \{T^{-\min_{\alpha\in\cA}h_\alpha^{m_k}}l_\alpha,\ldots,T^{\min_{\alpha\in\cA}h_\alpha^{m_k}}l_\alpha\}
$$
are left ends of some levels of towers for $\mathcal{R}^{m_{k+r}}(T)$. Therefore, by~\eqref{wys0} and by~\eqref{pigeon}, we obtain
\begin{align*}
\min\cP_{\min_{\alpha\in\cA}h_\alpha^{m_k},j}&\geq \min_{\alpha\in\cA} |I_\alpha^{m_{k+r}}|=|I_{\beta}^{m_{k+r}}|\\
&\geq \frac{1}{dC^2h_{\beta}^{m_{k+r}}}\geq\frac{1}{(dC^2)^{r+1}\min_{\alpha\in\cA}h_\alpha^{m_k}},
\end{align*}
where $\beta\in\cA$ is chosen so that the equality in the first line of the above expression holds, so it suffices to take $D=(dC^2)^{r+2}$ to get~\eqref{hmk}.
Now, we will prove that there exists $ D'>0$ such that for every $\alpha\in \cA$, every $n\in \N$ and $0\leq j\leq n-1$ 
\begin{equation}\label{inam} 
 \max \cP^\alpha_{n,j}\leq \frac{ D'}{n}.
\end{equation}
We claim that it suffices to show that there exists $D'>0$ such that
\begin{equation}\label{orb3a}
 \max\mathcal{P}(\{T^i x : 0\leq i\leq 2\max_{\alpha\in\cA}h_\alpha^{m_{k+r}}-1 \})\leq \frac{D'}{2dC^2 \max_{\alpha\in\cA}h_\alpha^{m_{k+r}}}
\end{equation}
holds for every $x\in [0,1)$ and $k\in\N$. Indeed, notice first that~\eqref{orb3a} means, in particular, that
\begin{equation}\label{orb3} 
\max \cP^\alpha_{2\max_{\alpha\in \cA} h_\alpha^{m_{k}},j}\leq \frac{D'}{2dC^2\max_{\alpha\in \cA}h_\alpha^{m_{k}}}
\end{equation}
holds for every $\alpha\in \cA$, $k\geq r+1$ and every $j\in \Z$. Next, fix $n\geq 2\max_{\alpha\in \cA}h_\alpha^{m_{r+1}}$, and let $k\geq r+1$ be unique such that 
$$
2\max_{\alpha\in \cA}h_\alpha^{m_{k}}\leq n<2\max_{\alpha\in \cA}h_\alpha^{m_{k+1}}.
$$
Then, by \eqref{orb3} (which holds for every $j\in \Z$) and~\eqref{wys0}, we obtain
\begin{multline*}
\max \cP^\alpha_{n,j}\leq \max \cP^\alpha_{2\max_{\alpha\in \cA}h_\alpha^{m_{k}},j}\\
  \leq \frac{D'}{2dC^2\max_{\alpha\in \cA}h_\alpha^{m_{k}}}\leq \frac{D'}{2\max_{\alpha\in \cA}h_\alpha^{m_{k+1}}}\leq\frac{D'}{n}.
\end{multline*}
Thus, we have shown that~\eqref{inam} holds for some $D'>0$ for $n\geq 2\max_{\alpha\in\cA}h_\alpha^{m_{r+1}}$. Adjusting $D'$ if necessary, we indeed obtain~\eqref{inam} (notice that the lef hand side of~\eqref{inam} takes only finitely many values when $k\leq r$). We will now prove~\eqref{orb3a}. Fix $x\in [0,1)$ and $k\in \N$. Notice that there exist $0\leq i_1<i_2\leq 2\max_{\alpha\in\cA}h_\alpha^{m_{k+r}}-1$ such that $T^{i_1}x,T^{i_2}x\in I^{m_{k+r}}$. Since all entries of $B(m_k,m_{k+r})$ are positive, it follows that for every $\alpha\in\cA$ there exists $i_1\leq i_\alpha<i_2$ such that $T^{i_\alpha}x\in I_\alpha^{m_k}$. In fact, the forward orbit of $T^{i_1}x$ of length $i_2-i_1$ visits every floor of each tower for $\mathcal{R}^{m_k}(T)$ at least once. Hence
\begin{multline*}
\max\mathcal{P}(\{T^i x : 0\leq i\leq 2\max_{\alpha\in\cA}h_\alpha^{m_{k+r}}-1 \})\leq 2\max_{\alpha\in\cA}|I_\alpha^{m_k}|\\
\leq\frac{2C}{\max_{\alpha\in\cA}h_\alpha^{m_k}}\leq \frac{2(dC^2)^rC}{\max_{\alpha\in\cA}h_\alpha^{m_{k+r}}}.
\end{multline*}
To obtain~\eqref{orb3a}, it suffices to take $D'=2(dC^2)^{r+1}C$.

The proof is now complete in view of~\eqref{em} and~\eqref{inam}, as clearly
$$
 \min\cP_{\min_{\alpha\in\cA}h_\alpha^{m_k},j}\leq \min\cP^\alpha_{\min_{\alpha\in\cA}h_\alpha^{m_k},j}
 $$
 and $\max \cP_{n,j}\leq \max \cP^\alpha_{n,j}$.
\end{proof}
\begin{uw}
Notice that the above proof yields, in particular, the following: whenever $T$ is an IET of bounded type then there exists a constant $c>0$ such that for any $n\in \N$ and any $x\in [0,1)$, for 
$$
\mathcal{P}_n(x):=\mathcal{P}(\{T^kx : 0\leq k\leq n-1\}),
$$
we have $\max \mathcal{P}_n(x) \leq \frac{c}{n}$. It turns out that also the following inequality is true: $\min\mathcal{P}_n(x) \geq \frac{1}{cn}$ for some constant $c>0$, all $x\in [0,1)$ and $n\geq 1$. Indeed, it was shown in~\cite{MR961737} that for any $x\in [0,1)$, $n\in \N$, we have either $T^nx=x$ or $|T^nx -x|\geq \min \mathcal{P}_{n+1,0}$. This, together with Proposition~\ref{rowy}, gives us the desired bound.
\end{uw}
\begin{uw}\label{inwer}
If follows immediately by Definition~\ref{de:2} that $T$ has balanced partition lengths if and only if $T^{-1}$ has this property. Therefore, in view of Proposition~\ref{as}, we obtain that $T$ is of bounded type if and only if $T^{-1}$ is of bounded type.
\end{uw}

\section{Special flow representation}\label{se:dach}
We will consider special flows with the base automorphism $T:\T\to \T$ being a minimal IET of bounded type and the roof function with symmetric logarithmic singularities at a subset of $\Sigma_{T}$. More precisely, let the functions $u,v:\R\to \R_+$  be given by
$$
u(x)=\frac{1}{x},\ v(x)=\frac{1}{1-x}\text{ for }x\in(0,1)
$$
and extended to $\R\setminus\Z$ in such a way that they are periodic of period $1$, i.e. for $x\in\R\setminus\Z$, $u(x) = u(\{x\})$ and $v(x) = v(\{x\})$, where $\{t\}$ denotes the fractional part of $t$. For $\alpha\in\cA$ let
$$
u_\alpha(x)=u(x-\ell_\alpha) \text{ and }v_{\alpha}(x)=v(x-r_\alpha).
$$
The roof function $f$ is such that its derivative is given by
\begin{equation}\label{roo}
f'(x)=-\sum_{\alpha\in\cA}C_\alpha^+u_\alpha(x)+\sum_{\alpha\in\cA}C_{\alpha}^-v_{\alpha}(x)+g(x),
\end{equation}
where the constants $C_\alpha^+,C_\alpha^-\geq 0$ are such that 
\begin{equation*}
\sum_{\alpha\in\cA}C_\alpha^+=\sum_{\alpha}C_\alpha^->0
\end{equation*}
and $g$ is a function of bounded variation and of class $C^2$ after restriction to 
$$
S:=\T\setminus(\{\ell_\alpha : C_\alpha^+> 0, \alpha\in\cA\}\cup \{r_\alpha: C_\alpha^->0, \alpha\in\cA\}).
$$
Thus the roof function $f$ is also of class $C^2$ after restriction to $S$. If the above conditions hold, we will say that~\emph{$f$ has symmetric logarithmic singularities at a subset of the discontinuities of $T$}.

\section{SWR-property}\label{se:swr}
In this section the central notion will be the \emph{SWR-property} introduced in~\cite{FKa}. Let $(X,d)$ be a $\sigma$-compact metric space, $\mathcal{B}$ the $\sigma$-algebra of Borel subsets of $X$, $\mu$ a Borel probability measure on $(X,d)$.  Let $\cT=(T_t)_{t\in\R}$ be an ergodic flow acting on $(X,\mathcal{B},\mu)$.
\begin{df}
Fix a compact set $P\subset \R\setminus\{0\}$ and $t_0>0$. The flow $\cT$ is said to have {\em $sR(t_0,P)$-property} if
\begin{align*}
&\text{for every $\vep>0$ and $N\in\N$ there exist $\kappa=\kappa(\vep)$, $\delta=\delta(\vep,N)$}\\
&\text{and a set $Z=Z(\vep,N)$ with $\mu(Z)>1-\vep$, such that}\\
&\text{for every $x,y\in Z$ with $d(x,y)<\delta$ and $x$ not in the orbit of $y$,}\\
&\text{there exist $M=M(x,y),L=L(x,y)\geq N$, $\frac{L}{M}\geq \kappa$}\\
&\text{and $p=p(x,y)\in P$}
\end{align*}
such that one of the following holds:
\begin{enumerate}[(i)]
\item
$\frac{1}{L}\left|\{n\in [M,M+L] : d(T_{nt_0}(x),T_{nt_0+p}(y))<\vep\}\right| >1-\vep$,
\item
$\frac{1}{L}\left|\{n\in [M,M+L] : d(T_{(-n)t_0}(x),T_{(-n)t_0+p}(y))<\vep\}\right| >1-\vep$.
\end{enumerate}
Moreover, $\cT$ has {\em SWR-property} (with the set $P$) if the set 
$$
\{t_0>0: \cT\text{ has }sR(t_0,P)\text{-property}\}
$$
is uncountable.
\end{df}
We will assume moreover that the flows under consideration are {\em almost continuous}. Recall that $\cT=(T_t)_{t\in\R}$ is said to be almost continuous if
\begin{align*}
&\text{for every $\vep>0$ there exists $X_\vep\subset X$ with $\mu(X_\vep)>1-\vep$}\\
&\text{such that for every $\vep'>0$ there exists $\delta'>0$ such that}\\
&\text{for every $x\in X$ we have $d(T_t(x),T_{t'}(x))<\vep'$ whenever $t,t'\in [-\delta',\delta']$}.
\end{align*}
\begin{tw}[\cite{FKa}]
Let $(X,d)$ be a $\sigma$-compact metric space, $\mathcal{B}$ the $\sigma$-algebra of Borel subsets of $X$, $\mu$ a Borel probability measure on $(X,d)$. Let $\cT$ be an almost continuous flow acting on $(X,\mathcal{B},\mu)$. If $\cT$ has SWR-property then $\cT$ enjoys FEJ-property.
\end{tw}

\subsection{SWR-property for special flows}
The following result giving a sufficient condition for SWR-property will be crucial for us.
\begin{pr}\label{cocy}
Let $(X,d)$ be a $\sigma$-compact metric space, $\mathcal{B}$ the $\sigma$-algebra of Borel subsets of $X$, $\mu$ a Borel probability measure on $(X,d)$. Let $T$ be an ergodic automorphism acting on $(X,\mathcal{B},\mu)$ and let $f\in L^1(X,\mathcal{B},\mu)$ be a positive function bounded away from zero. Let $\cT=(T_t^f)_{t\in\R}$ be the corresponding special flow.\footnote{Notice that such flows satisfy the almost continuity condition.} Let $P\subset \R\setminus\{0\}$ be a compact set. Assume that
\begin{align*}
&\text{for every $\vep>0$ and $N\in \N$ there exist $\kappa=\kappa(\vep)$, $\delta=\delta(\vep,N)$}\\
&\text{and a set $X'=X'(\vep,N)$ with $\mu(X')>1-\vep$, such that}\\
&\text{for every $x,y\in X'$ with $0<d(x,y)<\delta$}\\
&\text{there exist $M=M(x,y),L=L(x,y)\geq N$ with $\frac{L}{M}\geq \kappa$}\\
&\text{and $p=p(x,y)\in P$}
\end{align*}
such that one of the following holds:
\begin{enumerate}[(i)]
\item\label{posi}
$d(T^nx,T^ny)<\vep\;\text{and }\; |f^{(n)}(x)-f^{(n)}(y)-p|<\vep\text{ for every }n\in [M,M+L]$,
\item\label{nega}
$d(T^{-n}x,T^{-n}y)<\vep\;\text{and }\;|f^{(-n)}(x)-f^{(-n)}(y)-p|<\vep\text{ for every }n\in [M,M+L]$.
\end{enumerate}
If $\gamma>0$ is such that the automorphism $T^f_\gamma$ is ergodic, then $\cT$ has the $sR(\gamma,P)$-property.
\end{pr}
The proof of Proposition~\ref{cocy} goes exactly by the same lines as the proof of Proposition 3.3 in~\cite{FKa}, where $T$ was assumed to be an ergodic isometry. We relax this assumption, requiring instead, in~\eqref{posi} and~\eqref{nega}, that $d(T^nx,T^ny)<\vep$ and $d(T^{-n}x,T^{-n}y)<\vep$, respectively. Notice that if $T$ is an isometry then these conditions are satisfied automatically provided that $\delta(\vep,N)<\vep$.

\section{Proof of Theorem~\ref{main}}\label{se:joker}
This section consists of two parts. First, in Section~\ref{idea}, we present the very general idea of the proof and the basic tools. The proof of Theorem~\ref{main} is included in Section~\ref{core-}: first we present the core of the proof, then the details are given.

\subsection{The idea of the proof}\label{idea}

Ratner's properties are based one two mechanisms: for two close points what we want to see is:
\begin{enumerate}[(A)]
	\item after some time we want their orbits to diverge by  $p\in P$, where $P$ is some fixed compact set ({\bf detecting the drift}),\label{A}
	\item we want them to stay $p$-drifted for an $\varepsilon$-proportion of time ({\bf keeping the drift}). \label{B}
\end{enumerate}
For smooth surface flows, divergence of orbits is produced by the singularities of the derivative. Therefore, to obtain \eqref{A} and \eqref{B}, we need a controlled way (determined by the set $P$) of approaching the singularities. A crucial observation is that once points get too close to a singularity, their distance ``explodes'', which we want to avoid, since set $P$ is compact. However, for IETs of bounded type, either looking forward or backward, the orbits do not get too close to the singularities. More precisely, we have the following lemma (the proof is included later):
\begin{lm}\label{dis}
Let $x\in [0,1)$. Then for any $\delta>0$, 
\begin{equation}\label{lm:1}
\#\left\{n\in\left[ -\frac{1}{8\delta c},\frac{1}{8\delta c}\right] : \min_{\alpha\in\cA}\|\ell_\alpha-T^nx\|<\delta\right\}\leq 1
\end{equation}
\end{lm}
This will give~\eqref{A}. To obtain~\eqref{B}, we use again that the IETs under consideration are of bounded type -- this time the phenomenon we observe is that after the points reach the neighborhood of some singularity at time $M$, they stay ``far away'' from all singularities for time interval of length $\varepsilon M$ and we may use estimates from the following result:

\begin{pr}[Proposition 4.1 in~\cite{MR2800723}, see also Proposition 3.1 in~\cite{MR2992998}]\label{canc}
Let $\pi=(\pi_0,\pi_1)$ be an admissble pair of bijections from $\cA$ to $\{1,\dots,d\}$ and let $f$ be a function satisfying \eqref{roo}. Then for almost every $\lambda\in\R_+^\cA$, $|\lambda|=1$ there exists a constant $M'$ and a sequence of induction times $(n_k)_{k\in\N}$ for the corresponding IET $T_{\pi,\lambda}$ such that for every $z\in I_\beta^{n_k}$ and every $0<r\leq h_j^{n_k}$ we have
\begin{equation}\label{est}
\left|f'^{(r)}(z)+\sum_{\alpha\in\cA}\frac{C_\alpha^+}{z_\alpha^{\ell}}-\sum_{\alpha\in\cA}\frac{C_\alpha^-}{z_\alpha^{r}}\right|\leq M'r,
\end{equation}
where
$$
z_\alpha^\ell=\min_{0\leq i<r}|T^iz-\ell_\alpha|^+\text{ and }z_\alpha^r=\min_{0\leq i<r}|r_\alpha-T^iz|^+
$$
(for $x\in\R$, $|x|^+$ is equal to $x$ if $x\geq 0$ and it is equal to $\infty$ if $x<0$, so that $1/|x|^+$ is equal to zero for $x<0$).
\end{pr}
\begin{uw}[cf. Remark 3.2 in~\cite{MR2992998}]\label{zamiast}
One can check that if $T$ is of bounded type, the estimate~\eqref{est} from the above proposition holds and, furthermore, one can take as $(n_k)_{k\in\N}$ the sequence $(m_k)_{k\in\N}$ associated with the Marmi-Moussa-Yoccoz acceleration of the Rauzy induction. 
\end{uw}

\subsection{Proof}\label{core-}
\subsubsection{Core of the proof}\label{core}
We will use Proposition~\ref{cocy}. Since $\cT$ is weakly mixing, this will be sufficient to show that $\cT$ has SWR-property. 
Let 
$$
\widetilde{C}:=\sum_{\alpha\in\cA}C_\alpha^++\sum_{\alpha\in\cA}C_{\alpha}^-.
$$
Without loss of generality we may assume that $C_{\pi_0^{-1}(1)}^+=1$. It follows by \eqref{roo}, that there exists $D>0$ such that for every $0<x<\frac{1}{2} \min \cP_{1,0}$ (see Definition \ref{de:2}) we have
\begin{equation}\label{bon}
|f'(x)+u_0(x)|<D
\end{equation}
(recall that $u_0(x)=\nicefrac1{\{x\}}$ for $x\not\in\Z$).
Let 
\begin{equation}\label{cP}
P:=\left[-H,-\frac{1}{1600c^4}\right]\cup\left[\frac{1}{1600c^4},H\right],
\end{equation}
where $c$ is as in Definition \ref{de:2} and 
$$
H:=2\left(\frac{M'C}{400c^4+1}+\widetilde{C}\right)\left(\frac{dC^3(400c^4+1)}{32c}+1\right),
$$
where $C$ is as in~\eqref{mmy} and $M'$ is as in Proposition~\ref{canc}.\footnote{We will tacitly use the inequalities from Definition~\ref{de:2} throughout the proof.}

Fix $\vep>0$ and $N\in\N$. Let
\begin{equation}\label{kap}
\kappa:=\min\left(\frac{16c}{C^3(400c^4+1)},\frac{8c\vep}{dC^2M'},\frac{2}{C^3(400c^4+1)},\frac{\vep}{2\widetilde{C}C} \right)
\end{equation}
and
\begin{equation}\label{e:delta}
\delta:=\min\left(\vep,\frac{1}{64c},\frac{1}{256c^3},\frac{\kappa}{512c^5N},\frac{1}{2c(400c^4+1)},\frac{3}{4000c^4D},\frac{\vep}{4D} \right).
\end{equation}
Let $X':=\T\setminus\bigcup_{n\in\Z, \alpha\in\cA}T^n{\ell_\alpha}$. Consider points $x,y\in X'$ with 
$$
\eta:=\|x-y\|\in(0,\delta).
$$
By Lemma~\ref{dis} (with $\delta=2\eta$), at least one of the following two conditions holds:
\begin{equation}\label{odl}
\text{ for every } n\in [0,\frac{1}{16\eta c}]\text{, }\min_{\alpha\in\cA} \| \ell_\alpha-T^{n}x\|\geq 2\eta,
\end{equation}

\begin{equation}\label{odl2}
\text{ for every } n\in [1,\frac{1}{16\eta c}]\text{, }\min_{\alpha\in\cA} \| \ell_\alpha-T^{-n}x\|\geq 2\eta.
\end{equation}
What remains to be checked is that the assumptions of Proposition~\ref{cocy} are indeed satisfied. Crucial here will be the following facts:
\begin{lm}\label{cond}
Assume \eqref{odl} holds. Then there exists $k\in \left[\frac{N}{\kappa},\frac{1}{32\eta c}\right]\cap \Z$ such that one of the following holds:
\begin{align}
&f^{(k+1)}(x)-f^{(k+1)}(y)\in P,\label{ka}\\
&f^{(k)}(x)-f^{(k)}(y)\in P.\label{ka2}
\end{align}
Moreover
$$
|f^{(n)}(T^Mx)-f^{(n)}(T^My)|<\vep\text{ for every }n\in [0,L]
$$
is true for:
\begin{align}
M:=k+1,\ L:=\lceil\kappa M\rceil &\text{ whenever}~\eqref{ka} \text{ holds},\label{div}\\
M:=\lfloor(1-\kappa)k\rfloor,\ L:=\lceil\kappa M\rceil &\text{ whenever}~\eqref{ka2} \text{ holds}.\label{div2}
\end{align}
\end{lm}
\begin{uw}
Under the assumption that~\eqref{odl} holds, depending whether~\eqref{ka} or~\eqref{ka2} is true, we set $p:=f^{(k+1)}(x)-f^{(k+1)}(y)$ or $p:=f^{(k)}(x)-f^{(k)}(y)$, respectively. It follows by the cocycle equality that~\eqref{posi} from Proposition~\ref{cocy} holds.
\end{uw}

\begin{lm}\label{cond2}  
Assume \eqref{odl2} holds. Then there exists $k\in \left[\frac{N}{\kappa},\frac{1}{32\eta c}\right]\cap \Z$ such that one of the following holds:
\begin{align}
&f^{(-k)}(x)-f^{(-k)}(y)\in P,\label{ka3}\\
&f^{(-k+1)}(x)-f^{(-k+1)}(y)\in P.\label{ka4}
\end{align}
Moreover,
$$
|f^{(-n)}(T^{-M}x)-f^{(-n)}(T^{-M}y)|<\vep \text{ for every }n\in [0,L]
$$
is true for:
\begin{align}
M:=k,\ L:=\lceil\kappa M\rceil\text{ whenever}~\eqref{ka3}\text{ holds},\\
M:=\lfloor(1-\kappa)k\rfloor,\ L:=\lceil\kappa M\rceil\text{ whenever}~\eqref{ka4}\text{ holds}.
\end{align}
\end{lm}
\begin{uw}
Under the assumption that~\eqref{odl2} holds, depending whether~\eqref{ka3} or~\eqref{ka4} is true, we set $p:=f^{(-k)}(x)-f^{(-k)}(y)$ or $p:=f^{(-k+1)}(x)-f^{(-k+1)}(y)$, respectively.  It follows by the cocycle equality that~\eqref{nega} from Proposition~\ref{cocy} holds.
\end{uw}
\begin{uw}
Note that the constants $M,L,p$ depend on which of the conditions~\eqref{ka},~\eqref{ka2},~\eqref{ka3},~\eqref{ka4} in Lemma~\ref{cond} and Lemma~\ref{cond2} is satisfied.
\end{uw}
\begin{uw}
The first part of Lemma~\ref{cond} and Lemma~\ref{cond2} is ``responsible'' for \emph{detecting the drift}, whereas the latter part of each of these two lemmas is ``responsible'' for \emph{keeping the drift}.
\end{uw}
To complete the proof of Theorem~\ref{main} we need to prove Lemma~\ref{dis}, Lemma~\ref{cond} and Lemma~\ref{cond2}. Since the proofs of the two latter lemmas go along the same lines, we will provide only the proof of Lemma~\ref{cond} only.

\subsubsection{Proof of Lemma~\ref{dis}}
Suppose that~\eqref{lm:1} does not hold.  Let 
$$
-\frac{1}{8\delta c}\leq n_1<n_2\leq\frac{1}{8\delta c}\text{ and }\beta_1,\beta_2\in\cA
$$
be such that
\begin{equation}\label{e:1}
\|\ell_{\beta_{i}}-T^{n_i}x\|<\delta\text{ for }i=1,2.
\end{equation}
Without loss of generality, we may assume that for $n_1<n<n_2$
\begin{equation}\label{a:1}
\min_{\alpha\in\cA}\|\ell_\alpha - T^nx \|\geq \delta.
\end{equation}
We may moreover assume that $T^{n_1}x\leq \ell_{\beta_{1}}$. It follows from~\eqref{a:1} that for $0\leq k\leq n_2-n_1$
$$
(T^{k}[T^{n_1}x,\ell_{\beta_1}])\cap \{\ell_\alpha:\alpha\in\cA\}=\emptyset,
$$
i.e.\ $T^k$ acts as a translation on $[T^{n_1}x,\ell_{\beta_1}]$ for $0\leq k\leq n_2-n_1$. Therefore and by~\eqref{e:1}, we have
\begin{equation}\label{e:2}
\|T^{n_2-n_1}\ell_{\beta_1}-T^{n_2}x\|=\|\ell_{\beta_1}-T^{n_1}x\|<\delta.
\end{equation}
It follows from~\eqref{e:1} and~\eqref{e:2} that $\|\ell_{\beta_2}-T^{n_2-n_1}\ell_{\beta_1}\|<2\delta$.
On the other hand,
$$
\|\ell_{\beta_2}-T^{n_2-n_1}\ell_{\beta_1}\|\geq \min\cP_{2(n_2-n_1),2(n_2-n_1)}\geq \frac{1}{2c(n_2-n_1)}\geq \frac{1}{2c\frac{1}{4\delta c}}=2\delta,
$$
which yields a contradiction and the result follows.

\subsubsection{Proof of Lemma~\ref{cond}}\label{core1}
We claim that one of the following holds:
\begin{enumerate}[(a)]
\item\label{a}
$|f^{(k+1)}(x)-f^{(k+1)}(y)|\geq \frac{1}{1600 c^4}$,
\item\label{b}
$|f^{(k)}(x)-f^{(k)}(y)|\geq \frac{1}{1600c^4}$.
\end{enumerate}
Moreover
\begin{equation}\label{ha}
|f^{(k+1)}(x)-f^{(k+1)}(y)|<\frac{H}{2}\text{ and } |f^{(k)}(x)-f^{(k)}(y)|<H.
\end{equation}
This, by definition of $P$, gives \eqref{ka} if~\eqref{a} holds, and \eqref{ka2} if~\eqref{b} holds. We will show later that~\eqref{ka} implies~\eqref{div} (the proof of the fact that~\eqref{ka2} implies~\eqref{div2} is analogous). We will need the following lemma (whose proof will be also given later):
\begin{lm}\label{ek} There exists $k\in \left[\frac{N}{\kappa},\frac{1}{32\eta c}\right]\cap \Z$ such that 
\begin{equation}\label{kal}
2\eta<T^kx<400\eta c^4.
\end{equation}
\end{lm}

\paragraph{Proof of the fact that either~\eqref{a} or~\eqref{b} holds.}
By \eqref{odl}, for every $\alpha\in\cA$, $\ell_\alpha\notin [T^nx,T^ny]$ (recall that $\|x-y\|=\eta$). In other words,
\begin{equation}\label{tran} T^n\text{ acts as a translation on }[x,y]\text{ for }n\in \left[0,\frac{1}{16\eta c}\right].
\end{equation} 
Let $k$ be as in~\eqref{kal}. Then, by~\eqref{tran},~\eqref{kal} and~\eqref{e:delta}, we have
 \begin{equation}\label{tky}
T^ky=T^kx+\eta\leq 400\eta c^4+\eta\leq \delta(400c^4+1)\leq\frac{1}{2c}\leq \frac{1}{2}\min \cP_{1,0}.
\end{equation}
Therefore and by \eqref{bon}, we obtain for some $\xi\in [T^kx,T^ky]$
\begin{multline}\label{fk}
\left|(f^{(k+1)}(x)-f^{(k+1)}(y))-(f^{(k)}(x)-f^{(k)}(y)) \right|=\left| f(T^{k}x)- f(T^{k}y) \right| \\
=\|x-y\||f'(\xi)|=\eta|f'(\xi)| \geq \eta(|u(\xi)|-D)=\eta\left(\frac{1}{\xi}-D\right).
\end{multline}
Since $\xi\leq T^ky\stackrel{\eqref{tky}}{\leq}\eta(400c^4+1)$, it follows that
\begin{multline}\label{efk} 
\left|(f^{(k+1)}(x)-f^{(k+1)}(y))-(f^{(k)}(x)-f^{(k)}(y)) \right|\\
\stackrel{\eqref{fk}}{\geq} \frac{1}{400c^4+1}-\eta D
\geq \frac{1}{400c^4+1}-\delta D\stackrel{\eqref{e:delta}}{\geq}\frac{1}{800c^4},
\end{multline}
whence indeed either~\eqref{a} or~\eqref{b} holds.
\paragraph{Proof of \eqref{ha}.}

Note first, that by the fact that $k\leq\frac{1}{\lfloor 32\eta c \rfloor}-1$ and by \eqref{tran}, there exists $z\in [x,y]$ such that
  \begin{equation}\label{18ha}
|f^{(k+1)}(x)-f^{(k+1)}(y)|=\|x-y\|\cdot |f'^{(k+1)}(z)|=\eta|f'^{(k+1)}(z)|.
\end{equation}
Moreover, by the choice of $k$, it follows from \eqref{tky} that \begin{equation}\label{18i}
T^kz\in[T^kx,T^ky]\subset [0,\eta(400c^4+1)].
\end{equation}
We will now apply Proposition~\ref{canc} and Remark~\ref{zamiast} for $T^{-1}$ (cf.\ Remark~\ref{inwer}). To make the text more readable, we will use the same notation as in Proposition~\ref{canc}, Remark~\ref{zamiast} and in Section~\ref{se:ba}, even though they are formulated for $T$, not $T^{-1}$. Let $\ell\in \N$ be unique such that 
\begin{equation}\label{18j}
\min_{\beta\in\cA}\lambda_\beta^{m_\ell}\leq  \eta(400c^4+1)<\min_{\beta\in\cA}\lambda_\beta^{m_{\ell+1}}.
\end{equation}
It follows from \eqref{18i} and \eqref{18j} that $T^kz\in I^{m_\ell}_{\pi_0^{-1}(1)}$. Let $\alpha_s\in\cA$, $s\geq 1$, be such that for every $S\geq 1$
$$
(T^{-1})^{(\sum_{s=0}^{S-1}h^{m_\ell}_{\alpha_s})}(T^kz)\in I^{m_\ell}_{\alpha_S}
$$
(we put $\alpha_0=\pi_0^{-1}(1)$). Let $S_0\geq 1$ be unique such that 
\begin{equation}\label{18k}
\sum_{s=0}^{S_0-1}h^{m_\ell}_{\alpha_s}< k\leq\sum_{s=0}^{S_0}h^{m_\ell}_{\alpha_s}.
\end{equation}
We get
\begin{multline*}
\frac{S_0}{\eta(400c^4+1)}\stackrel{\eqref{18j}}{\leq} \frac{S_0}{\min_{\beta\in\cA}\lambda_\beta^{m_\ell}}\stackrel{\eqref{pigeon}}{\leq} S_0 dC^2 \max_{\beta\in\cA}h_\beta^{m_\ell}\\
\stackrel{\eqref{balancedtimes}}{\leq} S_0 dC^3 \min_{\beta\in\cA}h_\beta^{m_\ell}\leq dC^3 \sum_{s=0}^{S_0-1}h_{\alpha_s}^{m_\ell}\stackrel{\eqref{18k}}{\leq} dC^3k\leq \frac{dC^3}{32\eta c},
\end{multline*}
where the last inequality follows by the definition of $k$. Therefore 
\begin{equation}\label{18l}
S_0\leq \frac{dC^3(400c^4+1)}{32c}.
\end{equation}
Now, we apply Proposition~\ref{canc} to $T^{-1}$, for every $0\leq s\leq S_0-1$ with 
$$
r_s:=h^{m_\ell}_{\alpha_s}\text{ and }z_s:=(T^{-1})^{\sum_{t=0}^{s-1}h^{m_\ell}_{\alpha_t}}(T^kz)\in I^{m_\ell}_{\alpha_s},
$$
and then also for 
$$
r_{S_0}:=k-(\sum_{t=0}^{S_0-1}h^{m_\ell}_{\alpha_t})<h^{m_\ell}_{\alpha_{S_0}}\text{ and }z_{S_0}:=(T^{-1})^{\sum_{t=0}^{S_0-1}h^{m_\ell}_{\alpha_t}}(T^kz)\in I^{m_\ell}_{\alpha_{S_0}}.
$$
We obtain
\begin{equation}\label{dw1}
\left|f_{T^{-1}}'^{(r_s)}(z_s)\right|\leq  M'r_s +\sum_{\alpha\in\cA}\frac{C_\alpha^+}{z_{s,\alpha}^\ell}+\sum_{\alpha\in\cA}\frac{C_\alpha^-}{z_{s,\alpha}^r},
\end{equation}
where
$$
z_{s,\alpha}^\ell=\min_{0\leq i<r_s}|(T^{-1})^{i}z_s-\ell_\alpha|^+\text{ and }z_{s,\alpha}^r=\min_{0\leq i<r_s}|r_\alpha-(T^{-1})^iz_s|^+.
$$
It follows from~\eqref{pigeon},~\eqref{balancedtimes} and~\eqref{18j} that for every $0\leq s\leq S_0$ we have
\begin{equation}\label{21a}
r_s\leq h_{\alpha_s}^{m_\ell}\leq\frac{1}{\lambda_{\alpha_s}^{m_\ell}}\leq \frac{C}{d\min_{\beta\in\cA }\lambda_\beta^{m_{\ell+1}}}\leq\frac{C}{d\eta (400c^4+1)}.
\end{equation}
Notice that each $T^{-i}z_s$ for $0\leq i<r_s$ is of the form $T^{n_i}z$ for some $0\leq n_i\leq \frac{1}{32\eta c}$ (recall that $z\in [x,y]$). 
Therefore, 
\begin{multline*}
|T^{n_i}z-\ell_\alpha|^+ \geq \|T^{n_i}z-\ell_\alpha\|\geq\|T^{n_i}x-\ell_\alpha\|-\|T^{n_i}x-T^{n_i}z\|\\
\stackrel{\eqref{tran}}{=}\|T^{n_i}x-\ell_\alpha\|-\|x-z\|
\stackrel{\eqref{odl}}{\geq} 2\eta-\eta=\eta,
\end{multline*}
i.e.\ 
\begin{equation}\label{zl}
z_{s,\alpha}^\ell=\min_{0\leq i<r_s}|T^{n_i}z-\ell_\alpha|^+ \geq \eta.
\end{equation}
Note that for any $w\in \T$, $\alpha\in \cA$, we have $|r_\alpha-w|^+\geq \|l_{\gamma}-w\|$, for $\gamma\in \cA$ such that $\pi_0(\alpha)+1=\pi_0(\gamma) \bmod d$. Thus, we also obtain $z_{s,\alpha}^r\geq \eta$.

By \eqref{dw1}, \eqref{21a} and by the definition of $H$,
\begin{equation}\label{osz}
\left|f_{T^{-1}}'^{(r_s)}(z_s)\right|\leq \frac{M'C}{d\eta(400c^4+1)}+\frac{\widetilde{C}}{\eta}= \frac{H}{2\eta}\left(\frac{dC^3(400c^4+1)}{32c}+1\right)^{-1}.
\end{equation}
Using inequalities~\eqref{osz} for $0\leq s\leq S_0$, the cocycle identity and~\eqref{18l}, we get
\begin{equation}\label{wre} 
|f'^{(k+1)}(z)|=|f_{T^{-1}}'^{(k+1)}(T^kz)|\leq\sum_{s=0}^{S_0}|f_{T^{-1}}'^{(r_s)}(z_s)|\leq \frac{H}{2\eta}.
\end{equation}
Hence, using (\ref{18ha}), we conclude that $|f^{(k+1)}(x)-f^{(k+1)}(y)|<\frac{H}{2}$. Moreover, 
\begin{multline*}
|f^{(k)}(x)-f^{(k)}(y)|\\
\leq|f^{(k+1)}(x)-f^{(k+1)}(y)|+|f(T^kx)-f(T^ky)|\leq \frac{H}{2}+\eta |f'(\theta_k)|,
\end{multline*}
for some $\theta_k\in [T^kx,T^ky]$. Since $\theta_k\geq T^kx\stackrel{(\ref{odl})}{\geq}2\eta$, it follows that 
$$
|f^{(k)}(x)-f^{(k)}(y)|\leq \frac{H}{2}+\frac{1}{2}<H.
$$  
This finishes the proof of \eqref{ha}.

From now on we will assume that~\eqref{a} holds, i.e.~\eqref{ka} is true (we will indicate the necessary modification needed in case when~\eqref{b} holds).

\paragraph{Proof of the fact that~\eqref{ka} implies~\eqref{div}.}
Suppose that~\eqref{ka} holds. Note that 
$$
0\leq M+L+1\leq (1+\kappa)M\leq 2\frac{1}{32\eta c}=\frac{1}{16\eta c},
$$
whence, in view of~\eqref{tran},
$$
\ell_\alpha\notin [T^{M+n}x,T^{M+n}y]\text{ for }\alpha\in\cA\text{ and }n\in [0,L-1],
$$
i.e. $f^{(n)}$ is differentiable on $[T^Mx,T^My]$. Therefore, for 
 $n\in[0,L]$, we have
\begin{equation}\label{20,5}
|f^{(n)}(T^Mx)-f^{(n)}(T^My)=\|x-y\||f'^{(n)}(\xi_n)|<\vep,
\end{equation}
for some $\xi_n\in [T^Mx,T^My]$.

Now, we proceed analogously to the proof of (\ref{ha}), i.e.\ we use again Proposition~\ref{canc} and Remark~\ref{zamiast}. This time, we will apply them to $T$. We will keep using the same notation as before, i.e.\ as it stands in Proposition~\ref{canc}, even though we do not work with $T^{-1}$ anymore.\footnote{If~\eqref{b} holds, we apply Proposition~\ref{canc} and Remark~\ref{zamiast} to $T^{-1}$.}

Let $t\in N$ be unique such that
 \begin{equation}\label{22}
 \min_{\beta\in\cA}h_\beta^{m_{t-1}}\leq L<\min_{\beta\in\cA}h_\beta^{m_t}.
 \end{equation}
For every $n\in[0,L]$, $\xi_n\in [T^{k+1}x,T^{k+1}y]$. Hence 
\begin{equation}\label{22a}
T^{-1}\xi_n\stackrel{\eqref{tran}}{\in}[T^kx,T^ky]\stackrel{\eqref{18i}}{\subset}[0,\eta(400c^4+1)]. 
\end{equation}
We will show that $[0,\eta(400c^4+1)]\subset I_{\pi_0^{-1}(1)}^{m_t}$. Indeed,
$$
|I_{\pi_0^{-1}(1)}^{m_t}|\stackrel{\eqref{pigeon}}{\geq}\frac{1}{dC^2}\cdot \frac{1}{h_{\pi_0^{-1}(1)}^{m_t}}\stackrel{\eqref{wys}}{\geq}\frac{1}{C^3}\cdot \frac{1}{\min_{\beta\in\cA}h_\beta^{m_{t-1}}}\stackrel{\eqref{22}}{\geq}\frac{1}{C^3L}
$$
and, on the other hand,
\begin{align}
\begin{split}\label{22b}
C^3L&=C^3\lceil \kappa(k+1)\rceil\leq C^3(\kappa(k+1)+1)\\
&\leq C^3\left(\frac{\kappa}{32\eta c}+1 \right)\stackrel{\eqref{e:delta}}{\leq}\frac{\kappa C^3}{16\eta c}\stackrel{\eqref{kap}}{\leq}\frac{1}{\eta(400c^4+1)}.
\end{split}
\end{align}
Therefore, using Proposition~\ref{canc} for $T^{-1}\xi_n\in I^{m_t}_{\pi_0^{-1}(1)}$, $0<n\leq h_{\pi_0^{-1}(1)}^{m_t}$, we obtain
\begin{equation}\label{22c}
\left|f'^{(n)}(T^{-1}\xi)-\sum_{\alpha\in\cA}\frac{C_\alpha^+}{\xi_{n,\alpha}^\ell}+\sum_{\alpha\in\cA}\frac{C_\alpha^-}{\xi_{n,\alpha}^r}\right|\leq M'n,
\end{equation}
where  
$$
\xi_{n,\alpha}^\ell=\min_{0\leq i<n}|T^i(T^{-1}\xi_n)-\ell_\alpha|^+\text{ and }\xi_{n,\alpha}^r=\min_{0\leq i<n}|r_\alpha-T^i(T^{-1}\xi_n)|^+
$$
for $\alpha\in\cA$. Fix $0<n\leq h_{\pi_0^{-1}(1)}^{m_t}$. Notice that
\begin{equation}\label{23z}
h_{\pi_0^{-1}(1)}^{m_t}\stackrel{\eqref{wys0}}{\leq} dC^2\min_{\beta\in\cA}h_\beta^{m_{t-1}}\stackrel{\eqref{22}}{\leq} dC^2L\stackrel{\eqref{22b}}{\leq}\frac{\kappa dC^2}{16\eta c}\stackrel{\eqref{kap}}{\leq}\frac{\vep}{2M'\eta}.
\end{equation}
Moreover, using Lemma~\ref{dis} for $\delta=(h_{\pi_0^{-1}(1)}^{m_t}8c)^{-1}$ and $x=T^{-1}\xi_n$, we obtain that there exists at most one $n_0\in [0,...,h_{\pi_0^{-1}(1)}^{m_t}]$ such that
\begin{equation}\label{23}
\min_{\alpha\in\cA}\|\ell_\alpha-T^{n_0}(T^{-1}\xi_n)\| <\frac{1}{h_{\pi_0^{-1}(1)}^{m_t}8c}. 
\end{equation}
Since
\begin{equation}\label{eq:JJJ}
\frac{1}{h_{\pi_0^{-1}(1)}^{m_t}8c}\stackrel{\eqref{23z}}{\geq}\frac{1}{8c}\frac{16\eta c}{\kappa dC^2}=\frac{2\eta}{\kappa dC^2}\stackrel{\eqref{kap}}{\geq}\eta(400c^4+1),
\end{equation}
it follows from~\eqref{22a} and~\eqref{23} that $n_0=0$.
Therefore,
$$
\text{for }n\in [0,L]\stackrel{\eqref{22}}{\subset} \left[0,\min_{\beta\in\cA}h_\beta^{m_t}\right]{\subset}[0,h_{\pi_0^{-1}(1)}^{m_t}]\text{ and }\alpha\in\cA\setminus\{\pi_0^{-1}(1)\},
$$
we have 
$$
\frac{1}{\xi_{n,\alpha}^\ell}\stackrel{(\ref{23})}{\leq}h_{\pi_0^{-1}(1)}^{m_t}8c\stackrel{(\ref{23z})}{\leq} \frac{\kappa dM^2}{2\eta}\stackrel{(\ref{kap})}{\leq} \frac{\vep}{4\widetilde{C}\eta}.
$$ 
In the same way,
$$
\frac{1}{\xi_{n,\alpha}^r}\leq \frac{\vep}{4\widetilde{C}\eta}\text{ for }\alpha\in\cA.
$$
Therefore, by the definition of $\widetilde{C}$, for every $n\in [0,h_{\pi_0^{-1}(1)}^{m_t})$, we have
\begin{equation}\label{25}\left|
\sum_{\alpha\in\cA}\frac{C_\alpha^+}{\xi_{n,\alpha}^\ell}+\sum_{\alpha\in\cA}\frac{C_\alpha^-}{\xi_{n,\alpha}^r}\right|<
\frac{\vep}{4\eta}.
\end{equation}
Hence, for every $n\in [0,h_{\pi_0^{-1}(1)}^{m_t})$, using  the fact that $\xi_{n,\pi_0^{-1}(1)}^\ell=T^{-1}\xi_n$, we have
\begin{align*}
|f'^{(n)}(\xi_n)|=&|f'^{(n+1)}(T^{-1}\xi_n)-f'(T^{-1}\xi_n)|\\
\leq& \left|f'^{(n+1)}(T^{-1}\xi_n)+\frac{1}{\xi_{n,\pi_0^{-1}(1)}^\ell}\right|+\left|-\frac{1}{\xi_{n,\pi_0^{-1}(1)}^\ell}-f'(T^{-1}\xi_n)\right|\\
\leq&\left|f'^{(n+1)}(T^{-1}\xi_n)+\sum_{\alpha\in\cA}\frac{C_\alpha^+}{\xi_{n,\alpha}^\ell}-\sum_{\alpha\in\cA}\frac{C_\alpha^-}{\xi_{n,\alpha}^r}\right|\\
&+\left|\sum_{\alpha\in\cA}\frac{C_\alpha^+}{\xi_{n,\alpha}^\ell}+\sum_{\alpha\in\cA}\frac{C_\alpha^-}{\xi_{n,\alpha}^r}\right|+\left|\frac{1}{T^{-1}\xi_n}+f'(T^{-1}\xi_n)\right|\\
&\!\!\!\!\!\!\!\!\!\!\!\!\!\!\! \stackrel{\eqref{22c},\eqref{25},\eqref{bon}}{\leq} M'(n+1)+\frac{\vep}{4\eta}+D\leq M'h_{\pi_0^{-1}(1)}^{m_t}+\frac{\vep}{4\eta}+D\\
\stackrel{(\ref{23z})}{\leq}&\frac{\vep}{2\eta}+\frac{\vep}{4\eta}+D\stackrel{(\ref{e:delta})}{\leq} \frac{\vep}{\eta}.
\end{align*}
This, together with (\ref{20,5}), completes the proof of \eqref{div}.

\subsubsection{Proof of Lemma~\ref{ek}}\label{core2}

Notice that we only need to prove the right inequality in \eqref{kal}, the left inequality follows immediately by \eqref{odl}. Note that for
$$
\mathcal{P}_{\lfloor \frac{1}{32\eta c}\rfloor,0}=\left\{T^{-k}\ell_\alpha\colon 0\leq k\leq \left\lfloor \frac{1}{32\eta c} \right\rfloor-1, \alpha\in\cA \right\},
$$
in view of~\eqref{e:delta}, we have
\begin{equation}\label{e:22}
\max(\mathcal{P}_{\lfloor \frac{1}{32\eta c}\rfloor,0})\leq \frac{c}{\lfloor\frac{1}{32\eta c}\rfloor}\leq \frac{c}{\frac{1}{32\eta c}-1}\leq \frac{c}{\frac{1}{64\eta c}}=64\eta c^2.
\end{equation}
Moreover, there exist $0\leq j_0\leq \lfloor \frac{1}{32\eta c} \rfloor-1$ and $\alpha_{0}$ such that
$$
(T^{-j_0}\ell_{\alpha_{0}},x]\cap \mathcal{P}_{\lfloor \frac{1}{32\eta c}\rfloor,0}=\emptyset.
$$
It follows that
\begin{equation}\label{e:4}
\text{$T^k$ act as a translation on $[T^{-j_0}\ell_{\alpha_0},x]$ for }0\leq k\leq \left\lfloor\frac{1}{32\eta c} \right\rfloor
\end{equation}
and, by~\eqref{e:22},
\begin{equation}\label{e:4a}
|T^{-j_0}\ell_{\alpha_0}-x|\leq \max(\mathcal{P}_{\lfloor \frac{1}{32\eta c}\rfloor,0})\leq 64\eta c^2.
\end{equation}
Notice also that for
$$
\cP^{\alpha_0}_{\lfloor\frac{1}{32\eta c} \rfloor,\lfloor\frac{1}{32\eta c}\rfloor- j_0-1}=\left\{T^{k-j_0}\ell_{\alpha_0} : 0\leq k\leq \left\lfloor\frac{1}{32\eta c} \right\rfloor-1\right\},
$$
 we have
$$
\max \cP^{\alpha_0}_{\lfloor\frac{1}{32\eta c} \rfloor,\lfloor\frac{1}{32\eta c} \rfloor-j_0-1}\leq \frac{c}{\lfloor\frac{1}{32\eta c}\rfloor}\leq 64\eta c^2
$$
(see~\eqref{e:22}), 
whence there exists $0\leq k_0\leq \lfloor\frac{1}{32\eta c}\rfloor-1 $ such that
\begin{equation}\label{e:5}
0<T^{k_0-j_0}\ell_{\alpha_0}\leq 64\eta c^2.
\end{equation}
If $k_0\geq {N}/{\kappa}$ then we set $k:=k_0$. Suppose now that $k_0<{N}/{\kappa}$. Since
$$
\min \cP_{\lfloor \frac{1}{128\eta c^3}\rfloor,\lfloor \frac{1}{128\eta c^3}\rfloor-1}\geq \frac{1}{c\lfloor\frac{1}{128\eta c^3} \rfloor}\geq\frac{1}{c\frac{1}{128\eta c^3}}=128\eta c^2>64\eta c^2,
$$
it follows from~\eqref{e:5} that 
\begin{equation}\label{agt}T^n\text{ acts on }[0,T^{k_0-l_0}\ell_{\alpha_0}]\text{ as a translation for }0\leq n\leq \lfloor\frac{1}{128\eta c^3}\rfloor.
\end{equation}
 Moreover, since for
$$
\cP^{\pi_0^{-1}(1)}_{\lfloor\frac{1}{128\eta c^3} \rfloor,\lfloor \frac{1}{128\eta c^3} \rfloor-1}=\left\{T^k0 : 0\leq k\leq \left\lfloor \frac{1}{128\eta c^3}\right\rfloor -1\right\}
$$
we have
$$
\max \cP^{\pi_0^{-1}(1)}_{\lfloor\frac{1}{128\eta c^3} \rfloor,\lfloor \frac{1}{128\eta c^3} \rfloor-1}\leq \frac{c}{\lfloor \frac{1}{128\eta c^3}\rfloor},
$$
we can find $0<k_1\leq \lfloor \frac{1}{128\eta c^3}\rfloor-1$ such that
\begin{equation}\label{b:1}
0<T^{k_1}0<\frac{c}{\lfloor \frac{1}{128\eta c^3} \rfloor}\leq \frac{c}{\frac{1}{128\eta c^3}-1}\stackrel{\eqref{e:delta}}{\leq} 256 \eta c^4.
\end{equation}
Using
$$
\{T^{k_1}0,0\}\subset \cP^{\pi_0^{-1}(1)}_{k_1+1,k_1}=\{T^n0 :0 \leq n\leq k_1\},
$$
we obtain
$$
T^{k_1}0\geq\min \cP^{\pi_0^{-1}(1)}_{k_1+1,k_1}\geq \frac{1}{c(k_1+1)}\geq\frac{1}{2ck_1}.
$$
This, together with~\eqref{b:1} and \eqref{e:delta}, implies that
\begin{equation}\label{grow}
k_1\geq \frac{1}{512\eta c^5}\geq\frac{1}{512\delta c^5}\geq \frac{N}{\kappa}.
\end{equation}
Moreover, using~\eqref{e:5}, \eqref{agt} and~\eqref{b:1}, we obtain
\begin{align}\label{ali:1}
T^{k_1+k_0-j_0}\ell_{\alpha_0}&\leq |T^{k_1+k_0-j_0}\ell_{\alpha_0}-T^{k_1}0|+T^{k_1}0\nonumber\\
&=T^{k_0-j_0}\ell_{\alpha_0}+T^{k_1}0\\
&<64\eta c^2+256\eta c^4.\nonumber
\end{align}
Notice that by the choice of $k_0$ and $k_1$, we have
\begin{align}\label{ali:2}
\begin{split}
k:=k_0+k_1&\leq \frac{N}{\kappa}+\left\lfloor \frac{1}{128\eta c^3} \right\rfloor-1\\
&\stackrel{\eqref{grow}}{\leq} \frac{1}{512\eta c^5}+\frac{1}{128\eta c^3}-1\stackrel{\eqref{e:delta}}{\leq} \left\lfloor\frac{1}{32\eta c} \right\rfloor-1.
\end{split}
\end{align}
Thus, using~\eqref{grow} and~\eqref{ali:2}, by \eqref{ali:1}, we have shown that there exists $k\in\N$
\begin{equation}\label{img}
\frac{N}{\kappa}\leq k\leq \left\lfloor \frac{1}{32\eta c}\right\rfloor-1\text{ and }
0<T^{k-j_0}\ell_{\alpha_0}\leq 64\eta c^2 +256\eta c^4.
\end{equation}
This, in view of~\eqref{e:4},~\eqref{e:4a} and~\eqref{img}, implies that
\begin{align}\label{ali:3}
\begin{split}
0<T^kx&\leq |T^kx-T^{k-j_0}\ell_{\alpha_0}|+T^{k-j_0}\ell_{\alpha_0}\\
&=|x-T^{-j_0}\ell_{\alpha_0}|+T^{k-j_0}\ell_{\alpha_0}\\
&\leq 64\eta c^2+64\eta c^2+256\eta c^4<400\eta c^4.
\end{split}
\end{align}
This finishes the proof of Lemma \ref{ek}, making also the proof of Theorem~\ref{main} complete.

\footnotesize

\def\cprime{$'$}

\bigskip
\footnotesize

\noindent
Adam Kanigowski\\
\textsc{Institute of Mathematics, Polish Acadamy of Sciences, \'{S}niadeckich 8, 00-956 Warszawa, Poland}\par\nopagebreak
\noindent
\textit{E-mail address:} \texttt{adkanigowski@gmail.com}

\medskip

\noindent
Joanna Ku\l aga-Przymus\\
\textsc{Institute of Mathematics, Polish Acadamy of Sciences, \'{S}niadeckich 8, 00-956 Warszawa, Poland}\\
\textsc{Faculty of Mathematics and Computer Science, Nicolaus Copernicus University, Chopina 12/18, 87-100 Toru\'{n}, Poland}\par\nopagebreak
\noindent
\textit{E-mail address:} \texttt{joanna.kulaga@gmail.com}

\end{document}